\documentclass[10pt]{amsart}

\usepackage{amsmath, amssymb, amsthm, enumerate, bbold, wasysym ,mathrsfs}
\usepackage{amsfonts}
\usepackage[colorlinks,citecolor=red,pagebackref,hypertexnames=false]{hyperref}
\input amssym.def
\input amssym.tex

\newtheorem{thm}{Theorem}[section]
\newtheorem*{thm*}{Theorem}
\newtheorem{cor}[thm]{Corollary}
\newtheorem*{cor*}{Corollary}
\newtheorem{lem}[thm]{Lemma}

\newtheorem*{con*}{Conjecture}
\newtheorem*{prob*}{Problem}

\theoremstyle{definition}
\newtheorem{defn}[thm]{Definition}

\theoremstyle{remark}

\newcommand{\vc}[1]{ \mathbf{#1}}

\begin{document}
\author{Mathav Kishore Murugan}
\title{Generalized roundness of vertex transitive graphs}
\date{\today}
\email{mkm233@cornell.edu}
\address{Center for Applied Mathematics, Cornell University, Ithaca,NY 14853}

\begin{abstract}
We study the generalized roundness of finite metric spaces whose distance matrix $D$
has the property that every row of $D$ is a permutation of the first row. The analysis
provides a way to characterize subsets of the Hamming cube $\{ 0, 1 \}^{n}
\subset \ell_{1}^{(n)}$ ($n \geq 1$) that have strict $1$-negative type. The result can
be stated in two ways: a subset $S = \{ \vc{x}_0,\vc{x}_1,\ldots,\vc{x}_k \}$ of the Hamming
cube $\{ 0, 1 \}^{n} \subset \ell_{1}^{(n)}$ has generalized roundness one if and only
if the vectors $\{ \vc{x}_1 - \vc{x}_0,\vc{x}_2 - \vc{x}_0,\ldots,\vc{x}_k - \vc{x}_0 \}$
are linearly dependent in $\mathbb{R}^n$. Equivalently, $S$ has strict $1$-negative type
if and only if the vectors $\{ \vc{x}_1 - \vc{x}_0,\vc{x}_2 - \vc{x}_0,\ldots,\vc{x}_k - \vc{x}_0 \}$
are linearly independent in $\mathbb{R}^n$.
\end{abstract}

\maketitle

\section{Introduction}
This paper investigates the generalized roundness of finite metric spaces whose distance matrix $D$
has the property that every row of $D$ is a permutation of the first row. Examples of such metric spaces include
all vertex-transitive graphs (such as the Peterson graph, Cyclic graphs, Platonic solids, Archimedean
Solids, $K_{n,n}$, $K_n$, Hamming cubes, Cayley graphs of groups and so on). Motivation for our work
comes from recent papers of Wolf \cite{Wolf} and S\'{a}nchez \cite{Sanc}. All relevant background information
is given in Section \ref{sec0}.

In Section \ref{sec1}, we examine the generalized roundness of finite vertex-transitive graphs.
Given a finite vertex transitive graph $G$ endowed with the ordinary path
metric $\rho$ we characterize the cases of equality in the $\wp$-negative type inequalities of
$G$ where $\wp$ denotes the supremal $p$-negative type of $(G,\rho)$.
In Section \ref{sec2}, we use this result to provide a simple algebraic characterization of all
subsets of the Hamming cube $\{ 0, 1 \}^{n} \subset \ell_{1}^{(n)}$that have strict $1$-negative type.
It is worth noting that all metric subspaces of $L_{1}$-spaces have $1$-negative
type. (See, for example, \cite{Lenn}.)

\section{Negative Type and Generalized Roundness}\label{sec0}
The generalized roundness or, equivalently, supremal $p$-negative type of metric spaces
has been studied extensively in relation to isometric, uniform and coarse embedding problems
\cite{Scho, Enfl, Dran, Kell, Lafo}. However, computing the generalized roundness of even
simple infinite metric spaces can prove to be a very difficult task. The situation for finite metric
spaces is less drastic, at least in principal, due to recent advances by Wolf \cite{Wolf}
and S\'{a}nchez \cite{Sanc}. The relevant definitions are as follows.

\begin{defn}\label{NTGRDEF} Let $p \geq 0$ and let $(X,d)$ be a metric space. Then:

(a) $(X,d)$ has $p$-{\textit{negative type}} if and only if for
all finite subsets $\{x_{1}, \ldots , x_{n} \} \subseteq X$ ($n \geq 2$) and all choices of real numbers $\eta_{1},
\ldots, \eta_{n}$ with $\eta_{1} + \cdots + \eta_{n} = 0$, we have:

\begin{eqnarray*}
\sum\limits_{1 \leq i,j \leq n} d(x_{i},x_{j})^{p} \eta_{i} \eta_{j} & \leq & 0.
\end{eqnarray*}

(b) $(X,d)$ has \textit{strict} $p$-{\textit{negative type}} if and only if it has $p$-negative type
and the inequality in (a) is strict whenever the scalar $n$-tuple $(\eta_{1}, \ldots , \eta_{n}) \not=
(0, \ldots, 0)$.

(c) We say that $p$ is a \textit{generalized roundness exponent} for $(X,d)$, denoted by $p \in gr(X)$
or $p \in gr(X,d)$, if and only if for all natural numbers $n$,
and all choices of points $a_{1}, \ldots , a_{n}, b_{1}, \ldots , b_{n} \in X$, we have:

\begin{eqnarray*}\label{grinq}
\sum\limits_{1 \leq k < l \leq n} \left\{ d(a_{k},a_{l})^{p} + d(b_{k},b_{l})^{p} \right\} & \leq &
\sum\limits_{1 \leq j,i \leq n} d(a_{j},b_{i})^{p}.
\end{eqnarray*}

(d) We say that $p$ is a \textit{strict generalized roundness exponent} for $(X,d)$
if and only if the inequality in (c) is always strict.

(e) The \textit{generalized roundness} of $(X,d)$, denoted by $\mathfrak{q}(X)$ or
$\mathfrak{q}(X,d)$, is defined to be the supremum of the set of all generalized roundness exponents
of $(X,d)$:
\begin{eqnarray*}
\mathfrak{q}(X,d) & = & \sup \{ p : p \in gr(X,d) \}.
\end{eqnarray*}
\end{defn}

It is only relatively recently that metric spaces of strict $p$-negative ($p \geq 0$) have been
studied systematically. The notion is particularly significant in the context of finite metric
spaces. Hjorth \textit{et al}.\ \cite{Hjo2, Hjo1} examined finite metric spaces
of strict $1$-negative type and provided a wealth of examples, including all finite metric trees.
Much of the impetus for the present work concerned determining which subsets of the Hamming cube
$\{ 0, 1 \}^{n} \subset \ell_{1}^{(n)}$ have strict $1$-negative type. Theorem \ref{hc} provides,
in purely algebraic terms, a definitive answer to this question and figures as one of the main
results of this paper. Other papers that examine questions pertaining to strict $p$-negative type
include \cite{Dous, Hlaw, Sanc, Wolf}.

It turns out that generalized roundness and supremal $p$-negative type are equivalent.
The equivalences stated below are proven in \cite{Lenn, Dous}.

\begin{thm}\label{LTWthm}
Let $p \geq 0$ and let $(X,d)$ be a metric space. Then the following conditions are equivalent:
\begin{itemize}
\item[(1)] $(X,d)$ has (strict) $p$-negative type.

\item[(2)] $p$ is a (strict) generalized roundness exponent for $(X,d)$.
\end{itemize}
\end{thm}

\section{Generalized roundness of vertex-transitive graphs}\label{sec1}
Let $(X,d)$ be a finite metric space with points $\{x_1,x_2,\ldots,x_n\}$. Define the $p$-distance matrix as
$D_p=\left[ d(x_i,x_j) \right]_{i,j}$. Let $H$ denote the hyperplane
$\{ \alpha \in \mathbb{R}^n : \langle \alpha, \mathbb{1} \rangle =0 \}$ where $\mathbb{1}$ denotes
the vector of all ones and $\langle \cdot\, , \cdot \rangle$ denotes the standard inner product.

\begin{thm}\label{vt}
Let $(X,d)$ be a finite metric space with associated $p$-distance matrix $D_p$ such that each row of
$D_1$ is a permutation of the first row of $D_1$. Then we have:
\begin{itemize}
\item[(1)] $\mathfrak{q}(X,d)=\inf \{ p \ge 0: \det (D_p)=0 \}$.

\item[(2)] If $ \mathfrak{q}(X,d)=q < \infty$, we have that $\vc{u}^T D_q \vc{u} =0 , \vc{u} \in H$ if and only if
$D_q \vc{u} = \vc{0}  $.
\end{itemize}
\end{thm}

\begin{proof}
(1) Let $X$ be a $n$-point metric space and $p \ge 0$ with $\det (D_p) \neq 0$. Note that by the condition on the
rows of $D_1$ we have that $ \mathbb{1}$ is an eigenvector of $D_p$ with eigenvalue $\lambda= \sum_{i=1}^n
d(x_1,x_i)^p > 0$. Consequently $D_p^{-1} \mathbb{1} = \frac{1}{\lambda} \mathbb{1}$ and therefore
$ \langle D_p^{-1} \mathbb{1}, \mathbb{1} \rangle =  \frac{n}{\lambda} \neq 0$. Thus $\det (D_p) \neq 0$
implies that $ \langle D_p^{-1} \mathbb{1}, \mathbb{1} \rangle =  \frac{n}{\lambda} \neq 0$. Therefore by
Corollary 2.4 of \cite{Sanc} we have that $\mathfrak{q}(X,d)=\inf \{ p \ge 0: \det (D_p)=0 \}$.

(2) ($\Leftarrow$)
 Let
$D_q \vc{u} = \vc{0}$. $\mathbb{1} \in C(D_q)$, since $\mathbb{1}$ is an eigenvector corresponding to
a non-zero eigenvalue. Since $\vc{u} \in N(D_q)$ and $\mathbb{1} \in C(D_q)$, we have that $\langle \vc{u},
\mathbb{1} \rangle =0 $. Thus $\vc{u} \in H$. Also $\vc{u}^T D_q \vc{u} =\vc{u}^T \vc{0} = 0$.

($\Rightarrow$) Let $\vc{u}^T D_q \vc{u} =0$ and $\vc{u} \in H$. Let $\lambda_1 \ge \lambda_2 \ge \ldots
\ge \lambda_n$ be the eigenvalues of $D_q$ and let $\vc{v_1},\vc{v_2},\ldots, \vc{v_n}$ be a corresponding
orthonormal system of eigenvectors (this is possible since $D_q$ is symmetric). Since all entries in $D_p$
are continuous functions of $p$, the characteristic polynomial has coefficients continuous in $p$. Since
the roots of polynomial are continuous function of the coefficients, we have that the eigenvalues of $D_p$
are continuous functions of $p$. Note that $D_0$ has eigenvalues $n-1, -1 , -1 , \ldots, -1$ and that
$q = \inf \{ p \ge 0: \det (D_p)=0 \}$. Note that $\mathbb{1}$ is an eigenvector of $D_p$ with eigenvalue
$\sum_{i=1}^n d(x_1,x_i)^p > 0$. Since determinant of $D_p$ is zero if and only if an eigenvalue is zero,
by continuity of eigenvalues over $p$ and the fact that the first eigenvalue of $n-1$ of $D_0$ continuously
increases as $p$ increases (it always corresponds to the eigenvector $\mathbb{1}$), we have that $\lambda_1 > 0$,
$\lambda_2 = 0$, $\lambda_i \le 0$ for $i=3,\ldots,n$. Let $s$ be such that $ \lambda_1 > 0$ ,$\lambda_2 =
\ldots = \lambda_s = 0$ and $0>\lambda_{s+1} \ge \lambda_{s+2} \ge \ldots \ge \lambda_n$. Note that we can
choose $\vc{v_1} = \frac{1}{\sqrt{n}} \mathbb{1}$. Since $\vc{v_1},\vc{v_2},\ldots, \vc{v_n}$ are orthonormal
we have $\{\vc{v_2},\ldots, \vc{v_n} \}$ is a basis for $H$. Therefore there exists scalars $a_2,\ldots,a_n$
such that $\vc{u} = \sum_{i=2}^n a_i \vc{v_i}$. Note that
\begin{eqnarray*}
0 &=& \vc{u}^T D_q \vc{u} \\
& =& \left( \sum_{i=2}^n a_i \vc{v_i}^T \right) \left( \sum_{i=2}^n a_i \lambda_i \vc{v_i} \right) \\
&= & \sum_{i=2}^n a_i^2 \lambda_i \\
& = &  \sum_{i=s+1}^n a_i^2 \lambda_i.
\end{eqnarray*}
Since $ \lambda_i < 0$ and $a_i^2 \ge 0$ for all $i=s+1,\ldots,n$, we have that $a_i=0$ for all $i=s+1,\ldots,n$.
Thus $\vc{u}=  \sum_{i=2}^s a_i \vc{v_i}$. Therefore we have $D_q \vc{u} = \sum_{i=2}^s a_i \lambda_i \vc{v_i} = \vc{0}$.
\end{proof}

\section{Negative Type and the Hamming Cube}\label{sec2}
To illustrate an application of Theorem \ref{vt}, we turn now to an analysis of the negative type of
subsets of the Hamming cube. In order to do this we need to introduce some notation and develop two
technical lemmas. The main result then appears as Theorem \ref{hc}. This theorem completely
characterizes the subsets of the Hamming cube that have strict $1$-negative type.

Define $H_n = \{ 0 ,1 \}^n \subset l_1^{(n)}$ to be the $n$-dimensional Hamming cube.
Let $H_n = \{ b_{n,i}: i=0,1,\ldots,2^n -1 \}$ where $b_{n,i} \in H_n$ is such that $b_{n,i}$ is
the $n$-digit binary representation of $i$. Let $D_n$ be a $2^n \times 2^n$ distance matrix of $H_n$
such that $D_n= [ d(b_{n,i},b_{n,j}) ] _ { 0 \le i,j \le 2^n -1}$. Let $O_n$ be the
$2^n \times 2^n$ matrix whose entries are all 1.

Let $ \mathbb{1}_{i,j}$, $0 \le j \le i$ denote a $2^i \times 1$ vector such that the first $2^j$
coordinates of $ \mathbb{1}_{i,j}$ is $1$, the second $2^j$ coordinates is $-1$, the third $2^j$
coordinates is $1$ and so on. Let $A_n$ be a $ (n+1) \times 2^n $ matrix where the rows of $A_n$
are $\mathbb{1}_{n,n}^T,\mathbb{1}_{n,n-1}^T,\ldots, \mathbb{1}_{n,0}^T$.

\begin{lem} \label{lm1}
\begin{itemize}
Let $n \in \mathbb{N}$. Then:
\item[(1)]
\begin{eqnarray*}
D_1 &=& \begin{bmatrix} 0 & 1 \\ 1 & 0 \end{bmatrix}, \text{ and } \\
D_{n+1} &=& \begin{bmatrix} D_n & D_n + O_n \\ D_n + O_n & D_n \end{bmatrix}.
\end{eqnarray*}

\item[(2)] $D_n \mathbb{1}_{n,n} = n 2^{n-1} \mathbb{1}_{n,n}$, and
$D_n \mathbb{1}_{n,i} = - 2^{n-1} \mathbb{1}_{n,i}$
for all $i=0,1,\ldots,n-1$. \\

\item[(3)] Let $V_n = \text{Span} \{ \mathbb{1}_{n,i} : i=0,1,\ldots,n \}$. Then $C(D_n) = V_n$. \\

\item[(4)] $N(A_n)=N(D_n) = V_n^ \perp$.
\end{itemize}
\end{lem}

\begin{proof} (1) Since $H_1 = \{0 ,1 \}$ under the $l_1$ metric, we have that
$D_1 = \begin{bmatrix} 0 & 1 \\ 1 & 0 \end{bmatrix}$. If both $i,j < 2^n$,
then $b_{n+1,i} = (0 , b_{n,i}) , b_{n+1,j}= (0, b_{n,j})$, and so
\begin{eqnarray*}
d(b_{n+1,i}, b_{n+1,j}) &=& |0- 0 | + d(b_{n,i},b_{n,j}) \\
&=&  d(b_{n,i},b_{n,j}).
\end{eqnarray*}
If $i < 2^n$ and $j \ge 2^n$, then $b_{n+1,i} = (0 , b_{n,i}) , b_{n+1,j}= (1, b_{n,j-2^n})$, and so
\begin{eqnarray*}
d(b_{n+1,i}, b_{n+1,j}) &=& |0- 1 | + d(b_{n,i},b_{n,j-2^n}) \\
&=& 1+ d(b_{n,i},b_{n,j-2^n}).
\end{eqnarray*}
If $i \ge 2^n$ and $j < 2^n$, then $b_{n+1,i} = (1 , b_{n,i-2^n}) , b_{n+1,j}= (0, b_{n,j})$, and so
\begin{eqnarray*}
d(b_{n+1,i}, b_{n+1,j}) &=& |1- 0 | + d(b_{n,i-2^n},b_{n,j}) \\
&=& 1+ d(b_{n,i-2^n},b_{n,j}).
\end{eqnarray*}
If both $i,j \ge 2^n$, then $b_{n+1,i} = (1 , b_{n,i-2^n}) , b_{n+1,j}= (1, b_{n,j-2^n})$, and so
\begin{eqnarray*}
d(b_{n+1,i}, b_{n+1,j}) &=& |1- 1 | + d(b_{n,i},b_{n,j}) \\
&=&  d(b_{n,i-2^n},b_{n,j-2^n}).
\end{eqnarray*}
The conclusion follows from the above set of $4$ equations.

(2) We prove this by induction. For $n=1$ the result holds because
\begin{eqnarray*}
\begin{bmatrix} 0 & 1 \\1 & 0 \end{bmatrix} \begin{bmatrix} 1 \\ 1\end{bmatrix} &=& \begin{bmatrix} 1  \\1 \end{bmatrix}, \text{ and} \\
\begin{bmatrix} 0 & 1 \\1 & 0 \end{bmatrix} \begin{bmatrix} 1 \\ -1\end{bmatrix} &=& - \begin{bmatrix} 1 \\-1 \end{bmatrix}.
\end{eqnarray*}
By the induction hypothesis we have $D_n\mathbb{1}_{n,n}= n 2^{n-1} \mathbb{1}_{n,n}$.
Note that $O_n\mathbb{1}_{n,n}=2^n\mathbb{1}_{n,n}$. Then,
\begin{eqnarray*}
D_{n+1} \mathbb{1}_{n+1,n+1} &=&  \begin{bmatrix} D_n & D_n + O_n \\ D_n + O_n & D_n \end{bmatrix}
\begin{bmatrix} \mathbb{1}_{n,n} \\ \mathbb{1}_{n,n} \end{bmatrix} \\
&=& \begin{bmatrix} 2 D_n\mathbb{1}_{n,n} + O_n\mathbb{1}_{n,n} \\
2 D_n\mathbb{1}_{n,n} + O_n\mathbb{1}_{n,n}  \end{bmatrix} \\
&=& \begin{bmatrix}  n 2^{n}\mathbb{1}_{n,n} + 2^n\mathbb{1}_{n,n} \\
n 2^{n}\mathbb{1}_{n,n} + 2^ n\mathbb{1}_{n,n}   \end{bmatrix} \\
&=&  (n+1 )2^{n} \mathbb{1}_{n+1,n+1}, \text{ and}
\end{eqnarray*}

\begin{eqnarray*}
D_{n+1} \mathbb{1}_{n+1,n} &=&  \begin{bmatrix} D_n & D_n + O_n \\ D_n + O_n & D_n \end{bmatrix}
\begin{bmatrix} \mathbb{1}_{n,n} \\ -\mathbb{1}_{n,n} \end{bmatrix} \\
&=& \begin{bmatrix}  D_n\mathbb{1}_{n,n} -D_n\mathbb{1}_{n,n} - O_n\mathbb{1}_{n,n} \\
D_n\mathbb{1}_{n,n} + O_n\mathbb{1}_{n,n} -D_n\mathbb{1}_{n,n} \end{bmatrix} \\
&=& \begin{bmatrix}  - 2^n\mathbb{1}_{n,n}   \\ 2^ n\mathbb{1}_{n,n}   \end{bmatrix} \\
&=&  - 2^{n} \mathbb{1}_{n+1,n}.
\end{eqnarray*}
For $i=0,\ldots,n-1$, by the induction hypothesis we have $D_n \mathbb{1}_{n,i} = - 2^{n-1} \mathbb{1}_{n,i} $.
Note that $O_n \mathbb{1}_{n,i} = \vc{0}$. Then,
\begin{eqnarray*}
D_{n+1} \mathbb{1}_{n+1,i} &=&  \begin{bmatrix} D_n & D_n + O_n \\ D_n + O_n & D_n \end{bmatrix}
\begin{bmatrix} \mathbb{1}_{n,i} \\ \mathbb{1}_{n,i} \end{bmatrix} \\
&=& \begin{bmatrix}  D_n\mathbb{1}_{n,i} + D_n\mathbb{1}_{n,i} + O_n\mathbb{1}_{n,i} \\
D_n\mathbb{1}_{n,n} + O_n\mathbb{1}_{n,n} + D_n\mathbb{1}_{n,n} \end{bmatrix} \\
&=& \begin{bmatrix}  - 2^n\mathbb{1}_{n,i}   \\ -2^ n\mathbb{1}_{n,i}   \end{bmatrix} \\
&=&  - 2^{n} \mathbb{1}_{n+1,i}.
\end{eqnarray*}
The result follows by induction.

(3) By the previous part $\mathbb{1}_{n,i}$ is an eigenvector of $D_n$ corresponding to a non-zero eigenvalue
for $i=0,1,\ldots,n$. Therefore $\mathbb{1}_{n,i} \in C(D_n)$ for $i=0,1,\ldots,n$. Thus $V_n \subseteq C(D_n)$.
It suffices to show that every column of $D_n$ is in $V_n$. We show this by induction. The result is true for
$n=1$ since $V_1= C(D_1)= \mathbb{R}^2$. Let $\vc{c}_{n,i}$ denote the $i$-th column of $D_n$ for $i=0,1,\ldots,2^n-1$.
By the induction hypothesis $ \vc{c}_{n,i} \in V_n$ for all $i=0,1,\ldots,2^n-1$, \emph{i.e.} there exists scalars
$a_{i,j}$ such that $ \vc{c}_{n,i}= \sum_{j=0}^{n} a_{i,j} \mathbb{1}_{n,i}$ for all $i=0,1,\ldots,2^n-1$.
Let $i < 2^n$. Then by part (1) we have
  \begin{eqnarray*}
  \vc{c}_{n+1,i} &=& \begin{bmatrix} \vc{c}_{n,i} \\ \vc{c}_{n,i} + \mathbb{1}_{n,n} \end{bmatrix}\\
  &=& \begin{bmatrix} \vc{c}_{n,i} \\ \vc{c}_{n,i} \end{bmatrix}+  \begin{bmatrix} \vc{0} \\ \mathbb{1}_{n,n} \end{bmatrix}\\
  &=& \sum_{j=0}^{n-1} a_{i,j} \mathbb{1}_{n+1,i}+ (a_{i,n} + 0.5) \mathbb{1}_{n+1,n+1} - 0.5 \mathbb{1}_{n+1,n}.
  \end{eqnarray*}
If $i \ge 2^n$, then by part (1) we have
\begin{eqnarray*}
  \vc{c}_{n+1,i} &=& \begin{bmatrix} \vc{c}_{n,i}+ \mathbb{1}_{n,n} \\ \vc{c}_{n,i}  \end{bmatrix}\\
  &=& \begin{bmatrix} \vc{c}_{n,i} \\ \vc{c}_{n,i} \end{bmatrix}+  \begin{bmatrix}\mathbb{1}_{n,n} \\\vc{0} \end{bmatrix}\\
  &=& \sum_{j=0}^{n-1} a_{i,j} \mathbb{1}_{n+1,i}+ (a_{i,n} - 0.5) \mathbb{1}_{n+1,n+1} + 0.5 \mathbb{1}_{n+1,n}.
\end{eqnarray*}

Thus $c_{n+1,i} \in V_n$ for all $i=0,1,\ldots,2^{n+1}-1$. The result follows by induction.

(4) Since $D_n$ is symmetric we have by part (3), $N(D_n)= V_n^\perp$. Since the rows of $A_n$ spans $V_n$,
we have $C(A_n^T) = V_n$. Therefore $N(A_n) = V_n^\perp$.
\end{proof}

Let $M_n$ be a $(n+1) \times (n+1)$ matrix, defined as
\[ M_n = \begin{bmatrix} 1 &0& 0& \ldots& 0 \\
1 &-2& 0& \ldots& 0\\
1& 0& -2& \ldots& \vdots \\
\vdots & 0 & 0 &\ddots  & 0 \\
1 & 0 &  \ldots & 0 & -2
\end{bmatrix}.
\]
Let $B_n$ be a $(n+1) \times 2^n$ matrix, whose $i$-th column is $ \vc{r}_i = \begin{bmatrix} 1 \\ b_{n,i} \end{bmatrix}$,
for $i=0,1,\ldots,2^n-1$.
\begin{lem} \label{lm2}
$M_n$ is invertible and $M_n B_n = A_n$.
\end{lem}

\begin{proof} Let $\vc{c}_1,\ldots,\vc{c}_{n+1}$ be the columns of $M_n$. Note that the standard basis
$\vc{e}_1,\ldots,\vc{e}_{n+1}$ can be obtained as $\vc{e_1} =  \vc{c}_1 + 0.5 ( \vc{c}_2 + \ldots + \vc{c}_{n+1})$,
$\vc{e}_i= - 0.5 \vc{c}_i$ for $i=2,\ldots,n+1$. Thus the columns of $M_n$ are linearly independent and hence
$M_n$ is invertible. Note that the $i$-th column of $A_n$ is $M_n \begin{bmatrix} 1 \\ b_{n,i} \end{bmatrix}$,
for $i=0,1,\ldots,2^n-1$. This shows $M_n B_n = A_n$.
\end{proof}

Let $\vc{q}_{i}$ be the columns of $A_n$ and $\vc{r}_{i}$ be the columns of $B_n$ for $i=0,1,\ldots,2^n -1 $.

\begin{thm} \label{hc}
Let $S= \{ \vc{x}_0,\vc{x}_1,\ldots,\vc{x}_k \} \subseteq H_n$. Then %$\mathfrak{q}(S)=1$
$S$ has generalized roundness one if and only if
$\{ \vc{x}_1 - \vc{x}_0,\vc{x}_2 - \vc{x}_0,\ldots,\vc{x}_k - \vc{x}_0 \}$ are linearly dependent in $\mathbb{R}^n$.
Equivalently, $S$ has strict $1$-negative type
if and only if $\{ \vc{x}_1 - \vc{x}_0,\vc{x}_2 - \vc{x}_0,\ldots,\vc{x}_k - \vc{x}_0 \}$
are linearly independent in $\mathbb{R}^n$.
\end{thm}

\begin{proof} Let $I \subseteq \{ 0 ,1 ,\ldots, 2^n -1 \}$ be such that $S= \{ b_{n,i} : i \in I \}$.
Since $S \subseteq l_1^{(n)}$ and $\mathfrak{q}(l_1^{(n)})=1$ we have that $\mathfrak{q}(S) \ge 1$.
Since $S$ is a finite set, by Corollaries 4.3 and 5.11 of \cite{Hlaw}, we have that
$\mathfrak{q}(S)=1$ if and only if $S$ is not of strict 1-negative type \emph{i.e} $\mathfrak{q}(S)=1$
if and only if there exists $\vc{u} = (u_0 , u_1 , \ldots ,u_{2^n -1})^T \in \mathbb{R}^{2^n}$ such that
$\vc{u} \neq \vc{0}$ and $\vc{u}^T D_n \vc{u} =0 $, $\vc{u} \in H$ with $u_i = 0$ for all $i \notin I$.
By Theorem \ref{vt} and Lemma \ref{lm1} (4), $\mathfrak{q}(S)=1$ if and only if $\vc{u} \in N(D_n)=N(A_n)$
where $\vc{u} \neq \vc{0}$, $u_i = 0$ for all $i \notin I$. Thus $\mathfrak{q}(S)=1$ if and only if
$\{ \vc{q}_i : i \in I \}$ is linearly dependent. By Lemma \ref{lm2}, $M_n^{-1} \vc{q}_i = \vc{r}_i =
\begin{bmatrix} 1 \\ b_{n,i} \end{bmatrix}$. Therefore $\mathfrak{q}(S)=1$ if and only if
$\{ \vc{r}_i : i \in I \}= \{  \begin{bmatrix} 1 \\ b_{n,i} \end{bmatrix} : i \in I \} =
\{   \begin{bmatrix} 1 \\ \vc{x}_i \end{bmatrix}: i = 0,1, \ldots,k \}$ is linearly dependent.
Suppose $\{ \vc{r}_i : i \in I \}$ is linearly dependent. Then there are scalars $a_0,a_1,\ldots,a_k$,
not all, zero such that
\begin{eqnarray*}
\vc{0}&=& \sum_{i=0}^k a_i \begin{bmatrix} 1 \\ \vc{x}_i \end{bmatrix} \\
&=&  \begin{bmatrix} \sum_{i=0}^k a_i  \\ \sum_{i=0}^k a_i  \vc{x}_i \end{bmatrix}.
\end{eqnarray*}
This forces $a_0 = -\sum_{j=1}^k a_j$. Consequently $a_1,\ldots,a_k$ can't be all zero.
Thus we get $\sum_{i=1}^k a_i ( \vc{x}_i - \vc{x}_0 )= \vc{0}$. The whole computation can be reversed as well.
This shows that $\{ \vc{q}_i : i \in I \}$ is linearly dependent if and only if
$\{ \vc{x}_1 - \vc{x}_0,\vc{x}_2 - \vc{x}_0,\ldots,\vc{x}_k - \vc{x}_0 \}$ are linearly dependent
in $\mathbb{R}^n$. This implies that $\mathfrak{q}(S)=1$ if and only if
$\{ \vc{x}_1 - \vc{x}_0,\vc{x}_2 - \vc{x}_0,\ldots,\vc{x}_k - \vc{x}_0 \}$ are linearly dependent in $\mathbb{R}^n$.

The equivalent statement is plainly evident from the preceding argument.
\end{proof}

\begin{cor}\label{cr1}
If $S \subseteq H_n$ is a set of cardinality greater than $n+1$, then $\mathfrak{q}(S)=1$.
\end{cor}

\begin{proof}
Any set of $n+1$ vectors is dependent in $\mathbb{R}^n$. The result follows by Theorem \ref{hc}.
\end{proof}

We know from Theorem 1 of \cite{Djok} that every tree can be isometrically embedded in a hypercube.
We present an alternate proof of a lower bound on the number of dimensions of the hypercube required
to isometrically embed a tree with $k$-points. The following Corollary was proved in \cite{Grah}
\begin{cor}\label{cr2}
Let $T_k$ be a tree with $k$ points.
If $T_k$ is isometrically embedded in $H_n$, then $n \ge k-1$.
\end{cor}

\begin{proof}
We prove this by contradiction. By a result of \cite{Dous} $\mathfrak{q}(T_k) > 1$.
If $n < k - 1$, then $k > n + 1 $. Since $T_k \subseteq H_n$ with $k > n+1$,
by Corollary \ref{cr1} we get that $\mathfrak{q}(T_n)=1$, a contradiction.
\end{proof}

The bound given by Corollary \ref{cr2} is tight. It is shown by Graham and Pollak \cite{Grah}
that there exists an isometric embedding of $T_k$ in $H_{k+1}$.

If a finite metric space has strict $p$-negative type, then it must have $q$-negative type for
some $q > p$ by \cite[Corollary 4.2]{Hlaw}. As $H_{n}$ has only finitely many subsets (of strict
$1$-negative type), we deduce one more corollary from Theorem \ref{hc}.

\begin{cor}\label{cr3}
Let $n \geq 1$ be given. Let $\mathscr{Q}$ denote the set of all subsets of $H_{n}$ that
have strict $1$-negative type. Then there exists a $q > 1$ such that $\mathfrak{q}(S) \geq q$
for all $S \in \mathscr{Q}$.
\end{cor}

\begin{proof}
Apply Theorems \ref{LTWthm} and \ref{hc} with $p = 1$.
\end{proof}

\section{Acknowledgement}
The research presented in this paper was undertaken and completed at the 2011
Cornell University \textit{Summer Mathematics Institute} (SMI). I would
like to thank the Department of Mathematics and the Center for Applied Mathematics
at Cornell University for supporting this project, and the National Science
Foundation for its financial support of the SMI through NSF grant DMS-0739338.
I also thank Anthony Weston for directing this research at the SMI.

\bibliographystyle{amsalpha}

\begin{thebibliography}{Dous}

\bibitem{Djok} D. J. Djokovi\'{c},
               \textit{Distance-preserving subgraphs of hypercubes}, J. Combin. Theory, (1973), 263--267.

\bibitem{Dous} I. Doust and A. Weston, \textit{Enhanced negative type for finite metric trees}, J. Funct. Anal.
               \textbf{254} (2008), 2336--2364.

\bibitem{Dran} A. N. Dranishnikov, G. Gong, V. Lafforgue and G. Yu,
               \textit{Uniform embeddings into Hilbert space and a question of Gromov}, Can. Math. Bull. \textbf{45} (2002), 60--70.

\bibitem{Enfl} P. Enflo, \textit{On a problem of Smirnov}, Ark. Mat. \textbf{8} (1969), 107--109.

\bibitem{Grah} R. L. Graham, H. O. Pollak,
                \textit{On embedding graphs in squashed cubes}, Graph Theory and Applications, (1972), 99--110.

\bibitem{Hjo1} P. G. Hjorth, S. L. Kokkendorff and S. Markvorsen,
              \textit{Hyperbolic spaces are of strictly negative type}, Proc. Amer. Math. Soc. \textbf{130} (2002), 175--181.

\bibitem{Hjo2} P. Hjorth, P. Lison\v{e}k, S. Markvorsen and C. Thomassen,
              \textit{Finite metric spaces of strictly negative type}, Lin. Alg. Appl. \textbf{270} (1998), 255--273.

\bibitem{Hlaw} H. Li, A. Weston,
	       \textit{Strict p-negative type of a metric space}, (2010), 529--545.

\bibitem{Kell} C. Kelleher, D. Miller, T. Osborn and A. Weston, \textit{Strongly non-embeddable metric spaces},
               Top. Appl. (2011), doi:10.1016/j.topol.2011.11.041.

\bibitem{Lafo} J. -F. Lafont and S. Prassidis, \textit{Roundness properties of groups}, Geom. Dedicata
               \textbf{117} (2006), 137--160.

\bibitem{Lenn} C. J. Lennard, A. M. Tonge and A. Weston, \textit{Generalized roundness and negative type},
               Mich. Math. J. \textbf{44} (1997), 37--45.

\bibitem{Sanc} S. S{\'a}nchez, \textit{On the supremal $p$-negative type of a finite metric space},
               J. Math. Anal. Appl. (2011), doi:10.1016/j.jmaa.2011.11.043.

\bibitem{Scho} I. J. Schoenberg,
               \textit{On certain metric spaces arising from Euclidean spaces by a change of metric and their imbedding in Hilbert space},
               Ann. Math. \textbf{38} (1937), 787--793.

\bibitem{Wolf} R. Wolf, \textit{On the gap of finite metric spaces of $p$-negative type},
               Lin. Alg. Appl. (2011), doi:10.1016/j.laa.2011.08.031.

\end{thebibliography}

\end{document}